\documentclass[11pt]{article}  

\usepackage{amsbsy,amsmath,amsthm,amssymb,color,verbatim,ulem,color}

 \usepackage{bm}
 \usepackage{fullpage,cancel}
 \usepackage{float}
 \usepackage{wrapfig}
 \usepackage{caption}
 \usepackage{subcaption}
 \usepackage{multicol}
 \usepackage{multirow}
 \usepackage{longtable}
 \usepackage{makecell}
 \usepackage{supertabular}
 \usepackage{tikz}
    \usetikzlibrary{positioning}
 \usepackage{enumitem}
 \usepackage{ulem}

\newtheorem{theorem}{Theorem}

\theoremstyle{definition}

\usepackage{thmtools}
\usepackage{thm-restate}

\usepackage[colorlinks = true,
linkcolor = blue,
urlcolor  = blue,
citecolor = blue,
anchorcolor = blue]{hyperref}
\usepackage{cleveref}

\newcommand{\PF}[1]{\mathrm{PF}_{#1}}
\newcommand{\dPF}[2]{\mathrm{PF}_{#1}(#2)}
\newcommand{\ToH}{\mathrm{TH}}

\begin{document}
	
	\title{On Parking Functions and\\ The Tower of Hanoi}
	\markright{Parking Functions and The Tower of Hanoi}
	\author{Yasmin Aguillon$^1$, Dylan Alvarenga$^2$, Pamela E. Harris$^3$,\\ Surya Kotapati$^3$, J. Carlos Mart\'{i}nez Mori$^4$, Casandra D. Monroe$^5$,\\ Zia Saylor$^3$, Camelle Tieu$^6$, and Dwight Anderson Williams II$^7$}
	\date{%
    $^1$Swarthmore College\\%
    $^2$Cal Poly Pomona\\%
    $^3$Williams College\\%
    $^4$Cornell University\\%
    $^5$The University of Texas at Austin\\%
    $^6$University of California, Irvine\\%
    $^7$MathDwight\\[2ex]%
    \today
}
	\maketitle
	
	\begin{abstract}
		The displacement of a parking function measures the total difference between where cars want to park and where they ultimately park. In this article, we prove that the set of parking functions of length $n$ with displacement one is in bijection with the set of ideal states in the famous Tower of Hanoi game with $n+1$ disks and $n+1$ pegs, both sets being enumerated by the Lah numbers. 
	\end{abstract}
	
	\section{Introduction.}
	Fix $n\in\mathbb{N}=\{1,2,3,\ldots\}$ and consider $n$ cars entering sequentially to park on a one-way street consisting of $n$ spots numbered consecutively from 1 to $n$.
	The cars enter the street one at a time
	and first drive to their preferred spot, where they park if unoccupied. 
	If a car's preferred spot is occupied, then it continues down the street, parking in the first available spot it encounters.
	For each $i\in[n]:=\{1,2,\ldots,n\}$, we let $a_i\in[n]$ be the preferred parking spot of car $i$; appropriately, $\alpha=(a_1,a_2,\ldots,a_n)\in[n]^n$ is called a preference vector. 
	If all cars can park given the preference vector $\alpha$, then we say that $\alpha$ is a \textit{parking function} of length $n$. 
	For example, $(3,1,1,3,2)$ is a parking function of length 5 since car 1 parks in spot 3, car 2 parks in spot 1, car 3 parks in spot 2, car 4 parks in spot 4, and car 5 parks in spot 5. 
	On the other hand, $(3,4,2,3)$ is not a parking function, as no car prefers spot 1, and hence the cars cannot all park in the four available spots.
	We remark that parking functions do not require that all cars park where they prefer; indeed, all that matters is that every car can park within the $n$ parking spots on the one-way street. 
	Throughout, we let $\PF{n}$ denote the set of all parking functions of length $n$ and $|\PF{n}|$ denote the number of parking functions of length $n$. Konheim and Weiss first introduced parking functions in 1966 in their study \cite{konheimOccupancyDisciplineApplications1966} of hashing functions. They established that $|\PF{n}|=(n+1)^{n-1}$ is the number of parking functions of length $n$ (in \textit{loc. cit. denoted $T_{n}$}).

	Gessel and Seo \cite{gesselRefinementCayleyFormula2006} classify parking functions based on the number of \textit{lucky} cars --- those that park in their preferred spot. 
	Note that parking functions that are permutations of $[n]$ are the luckiest of all, as every car parks in its preferred spot. 
	In contrast, unless $n=1$, the parking function consisting of all ones is as unlucky as can be, since no car except the first can park in its preferred spot. For more on parking functions, we point the interested reader to \cite{yan2015parking}.

	An alternative to the counting of lucky cars, our work is motivated by a related but separate statistic measuring a type of holistic luckiness of a parking function: the \textit{displacement of a parking function}. To begin, we note that the \textit{displacement} of car $i$ under the parking function $\alpha$ is defined as the difference between the number of the spot where a car actually parks and its preferred spot. If this value is $k_i$, then we say that car $i$ has been \textit{bumped} $k_i$ spots down the road. 
	So the displacement of a parking function $\alpha$, which we denote by $d(\alpha)$, is the total amount all cars are bumped under this parking function. Precisely, the displacement of a parking function $\alpha$ is given by
	$d(\alpha):=\sum_{i=1}^nk_i$,
	where $k_i$ denotes the displacement of car $i$ when parking under $\alpha$.
	In this way, 
	the displacement of a parking function $\alpha$ is an integral value that quantifies how ``unlucky'' $\alpha$ is. 
	If $\alpha$ is a permutation of $[n]$, then $d(\alpha)=0$.
	In all other cases, 
	the displacement satisfies $d>0$. 
	In particular, if $\alpha$ is the preference vector consisting of all ones, then $d(\alpha)=\frac{n(n-1)}{2}$.
	
	We remark that displacement of a parking function does not determine the number of (un)lucky cars in a parking function. 
	For example, a parking function $\alpha$ may satisfy $d(\alpha)=7$, but this does not provide any indication as to whether a single car has a displacement of seven or if seven cars have a displacement of one. Nonetheless, in comparing the parking functions $(1, 2, 2)$ and $(1, 2, 1)$, each with one unlucky car, we see the former has displacement one, whereas the latter has displacement two, illustrating the distinction between measuring ``luckiness'' of a parking function and counting unlucky cars. In what follows, we let $\dPF{n}{d}$ denote the set of parking functions of length $n$ that have exactly $d$ bumps (displacement $d$) and $|\dPF{n}{d}|$ denote the number of such parking functions.
	
	Our main contribution in this paper provides a connection between the displacement statistic on parking functions and the Tower of Hanoi --- a mathematical puzzle invented by the French mathematician \'{E}douard Lucas in 1883 \cite{hinzTowerHanoiMyths2018}.
	We focus on the $(n+1)\times (n+1)$ Tower of Hanoi game, where there are $n+1$ disks labeled 0 to $n$ by increasing size and $n+1$ pegs labeled 0 to $n$ from left to right.
	In the game's starting position, the $n+1$ disks begin stacked on the first peg and are arranged in descending order, with disk $n$ on the bottom and disk 0 on the top.
	We refer to the first peg (labeled 0) as the \textit{source} peg and the last peg (labeled $n$) as the \textit{destination} peg.
	Moreover, all pegs in between the source and destination pegs are the \textit{interior} pegs.

	The objective of this game is to move all $n+1$ disks from the source peg to the destination peg, ending with the disks in the same order as in the starting position. Thus $n \geq 2$, as there must be at least three pegs/disks to consider when the number of pegs equals the number of disks.
	We illustrate the starting and ending positions for $n=3$ in Figure~\ref{fig: positions}.
	
	\begin{figure}[ht]
		\centering
		\begin{subfigure}{.45\linewidth}
			\centering
			\includegraphics[width=0.89\linewidth]{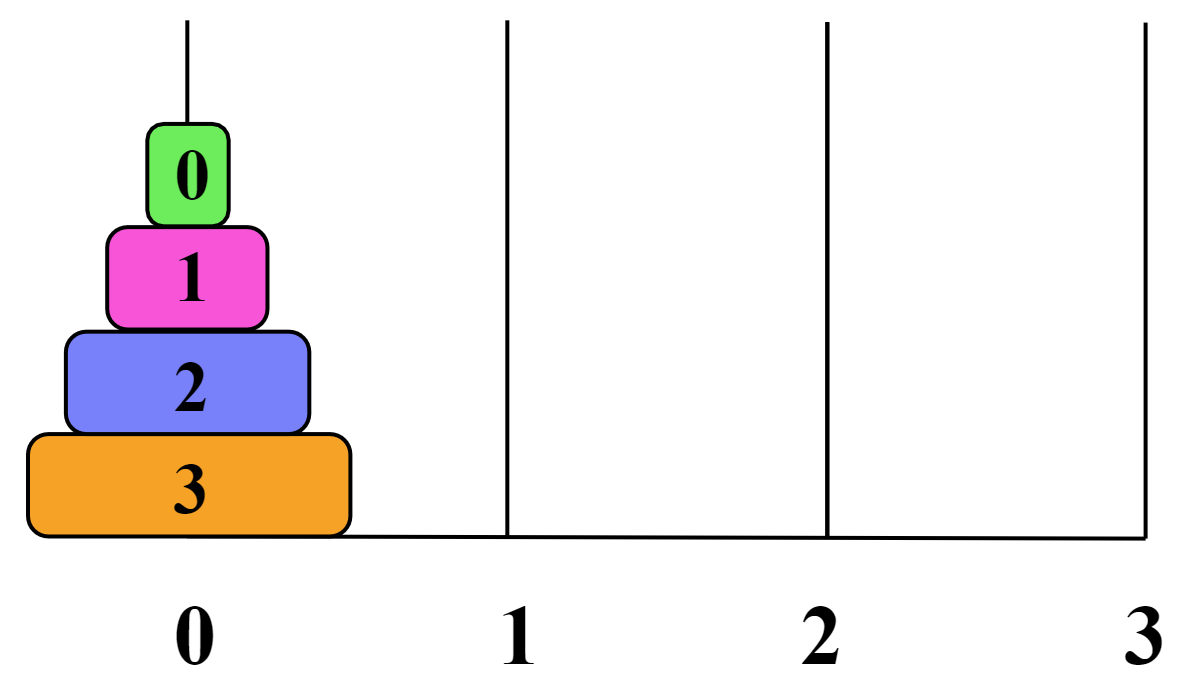}
			\caption{Starting position.}
			\label{fig: positions a}
		\end{subfigure}
		\begin{subfigure}{.45\linewidth}
			\centering
			\includegraphics[width=0.89\linewidth]{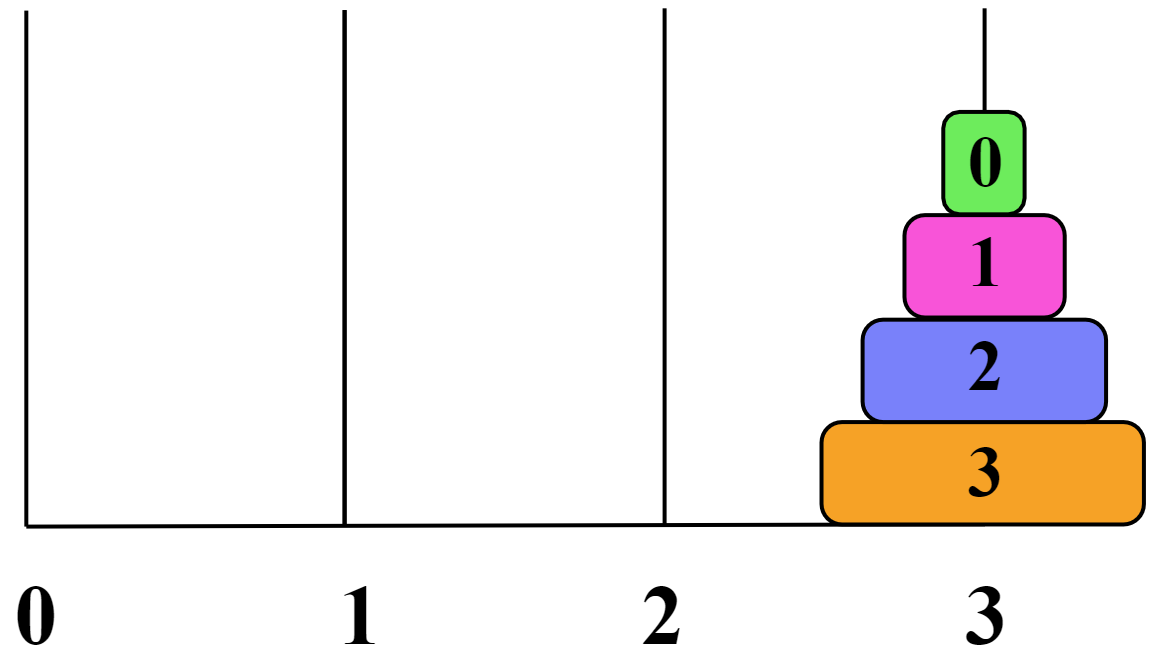}
			\caption{Ending position.}
			\label{fig: positions b}
		\end{subfigure}
		\caption{Starting position (\ref{fig: positions a}) and ending position (\ref{fig: positions b}) in the $(3 + 1)\times (3 + 1)$ Tower of Hanoi game.}
		\label{fig: positions}
	\end{figure}

	Moving a disk from one peg to another must follow two rules. 
	First, we can only move one disk at a time. 
	Second, we can never place a larger disk on top of a smaller one. 
	At the completion of each move, we refer to the arrangement of the disks on the pegs as a \textit{state}.
	In this way, we also refer to the initial and ending positions as the initial and ending states of the game, respectively. We use the term \textit{strategy} for the collection of moves made during a single game or, equivalently, the collection of states reached throughout a game.
	A particular intermediate state (between initial and ending states) is an \textit{ideal state}, which is an arrangement of $n+1$ disks with the following properties: 
	\begin{itemize}
		\item the largest disk (labeled $n$) is alone on the source peg (labeled 0),
		\item the destination peg (labeled $n$) is the only empty peg, and 
		\item the remaining $n$ disks are on the $n-1$ interior pegs in such a way that each interior peg has at least one disk; hence, exactly one peg has two disks.
	\end{itemize} 
	
	We let $\ToH_{n+1}$ denote the set of all distinct ideal states in the $(n+1)\times(n+1)$ Tower of Hanoi game. The ideal states in a $(n+1)\times(n+1)$ Tower of Hanoi game are significant in achieving an \textit{optimal strategy}, which is a winning strategy with the minimum number of moves possible: If a strategy places a game in an ideal state after exactly $n+1$ moves, then such a strategy guarantees that the remainder of the game can be won after $n+2$ more moves, $2n+3$ moves being the minimum number of moves to win \cite[see the proof of Proposition 2]{klavvzar2005hanoi}. 
	Furthermore, any optimal strategy of an $(n+1)\times(n+1)$ Tower of Hanoi game will necessarily result in an ideal state after $n+1$ moves. 
	We remark that determining optimal strategies for Tower of Hanoi games is a very active area of research with a history of long-standing problems \cite{bousch2014quatrieme,ottodunkelProblemsSolutions1941} and applications \cite{simon1975functional}. For more, we refer the reader to \cite{hinzTowerHanoiMyths2018}.

	We now present our main result connecting parking functions of displacement one, those almost as lucky as permutations, and ideal states of the Tower of Hanoi game, which can guarantee a minimal-move win.
	\begin{theorem}\label{thm:main}
		Regard $\emptyset = \ToH_{2} = \dPF{1}{1}$. The set $\ToH_{n+1}$ of distinct ideal states in the $(n+1)\times (n+1)$ Tower of Hanoi game and the set $\dPF{n}{1}$ of parking functions of length $n$ with displacement one are in bijection. Moreover, they are enumerated by the Lah numbers (\textcolor{blue}{\href{http://oeis.org/A001286}{A001286}}).
	\end{theorem}
	We organize the present paper as follows:
	The results of Section \ref{sec: pfn1}
	give characterizations of the sets
	$\dPF{n}{1}$ and $\ToH_{n+1}$. In Section \ref{sec: bijection}, we define an explicit bijection between $\dPF{n}{1}$ and $\ToH_{n+1}$, establishing Theorem \ref{thm:main}, and as a corollary we note that both sets are enumerated by the Lah numbers.
	
	\section{Characterizing the elements of \texorpdfstring{$\dPF{\lowercase{n}}{1}$}{displacement one parking functions} and \texorpdfstring{$\ToH_{\lowercase{n}+1}$}{ideal states}.}\label{sec: pfn1}

	We begin by giving a characterization for the elements of $\dPF{n}{1}$.
	\begin{theorem}{\label{thm: pf}}
		Let $\alpha=(a_1,\ldots,a_n)\in[n]^n$. Then $\alpha\in PF_n(1)$ if and only if
		$\alpha$ 
		satisfies:
		\begin{enumerate}
			\item there exists $j\in[ n-1]$ and distinct $k,k'\in [n]$ such that $a_k=a_{k'}=j$ and $a_i\neq j$ for any other index $i\neq k, k'$ and 
			\item $
			\{a_i:i\in[n] \mbox{ and }i\neq k, k'\}
			=
			[n]\setminus\{j,j+1\}$.
		\end{enumerate}
	\end{theorem}
	\begin{proof}
		Begin with the case $\alpha=(a_1,\ldots,a_n)\in\dPF{n}{1}$. 
		This indicates that there is exactly one unlucky car $k \in [n]$, and it has displacement one.
		Suppose the unlucky car prefers spot $j$, i.e., $a_k = j$.
		Note that the unlucky car must share a preference with exactly one of the lucky cars. 
		Suppose car $k' \in [n]$ is the car with this preference, i.e., $a_{k} = a_{k'} = j$ and $k \neq k'$.
		Observe that $j\in[n-1]$, otherwise, if $j=n$, then the unlucky car would be unable to park. Hence condition 1 is satisfied.
		Moreover, since $\alpha\in\dPF{n}{1}$, we know that the unlucky car is displaced by one spot and parks in spot $j+1$, specifically. Had any car preferred spot $j+1$, then the unlucky car would have been bumped down the road more than one spot, or there would have been another unlucky car under $\alpha$. Either scenario would imply $d(\alpha) \geq 2$, which is in contradiction to $\alpha\in\dPF{n}{1}$. Since the $n-2$ cars unequal to $k$ or $k'$ park in the $n-2$ spots different from $j$ or $j+1$, condition 2 is satisfied.
		This completes the forward direction.
		
		For the converse it suffices to show that such a preference vector lies in $\dPF{n}{1}$.
		Conditions 1 and 2 guarantee that only one entry in $\alpha$ is repeated, and this value is repeated exactly once. 
		This means that there are $n-2$ cars preferring separate parking spots, and each of these cars parks in accordance to its preference, contributing zero to the displacement.
		We now consider the two distinct cars $k, k' \in [n]$ that have the same preference $j$. 
		Assume without loss of generality that $k'<k$, implying that car $k'$ parks in spot $j$ and contributes zero to the displacement. 
		Now condition 2 ensures that spot $j+1$ is an available parking spot. Consequently, car $k$, being the other car with preference $j$, finds spot $j$ occupied by car $k'$ and parks in spot $j+1$, thereby contributing one to the displacement. Thus $\alpha\in\dPF{n}{1}$.
	\end{proof}

	Now turning to the Tower of Hanoi problem, we begin by setting the stage for describing ideal states. A state is denoted by \[x=(x_0,x_1,\ldots,x_{n})\] where, for each $0\leq i\leq n$, disk $i$ is located on peg $x_i$. 
	We make the assumption that at any point a vector $x$ with repeated entries satisfies the condition that the disks on the same peg are ordered with the smallest on the top and the largest disk on the bottom.
	For example, $x=(0,0,0,0)$ means all disks are on peg 0 (in order from smallest on top to largest at the bottom).
	Figure \ref{fig: anidealstate} provides a visualization of the ideal state $(2,2,1,0)$ of the $4 \times 4$ Tower of Hanoi game.
	
	\begin{figure}[ht]
		\centering
		\includegraphics[width=0.4\linewidth]{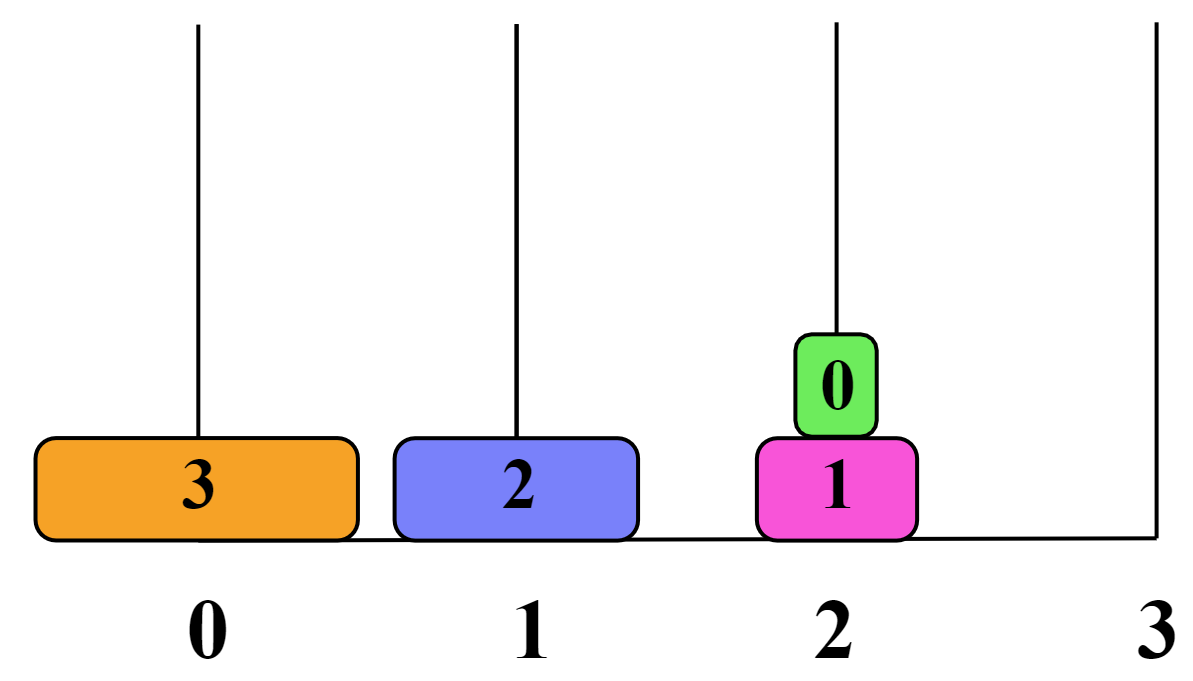}
		\caption{The ideal state $(2,2,1,0)$ in the $(3+1) \times (3+1)$ Tower of Hanoi game.}
		\label{fig: anidealstate}
	\end{figure}
	
	Let $n=3$. In Figure~\ref{fig: 4x4ideal}, we illustrate the six different ideal states and how we arrive at these arrangements in four moves.  In Figure \ref{fig: treegraph}, we use the vector notation for states to give a new visualization of Figure \ref{fig: 4x4ideal}.
	%\newpage
	\begin{figure}[ht]
		\centering
		\includegraphics[width=0.73\linewidth]{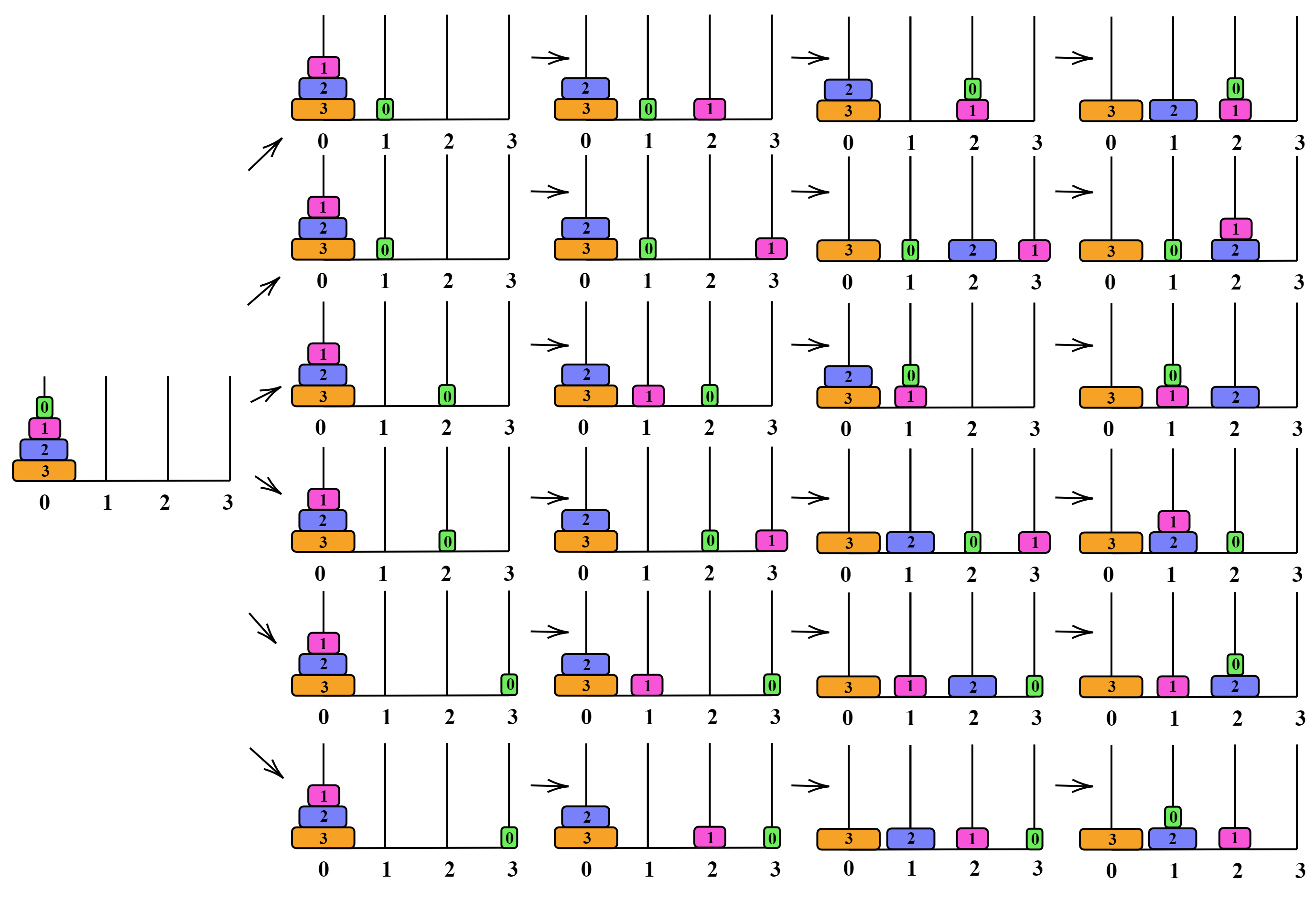}
		\caption{Ideal states in $\ToH_{3+1}$.}
		\label{fig: 4x4ideal}
	\end{figure}
	
	\begin{figure}[ht]
		\centering
		\includegraphics[scale = 0.24]{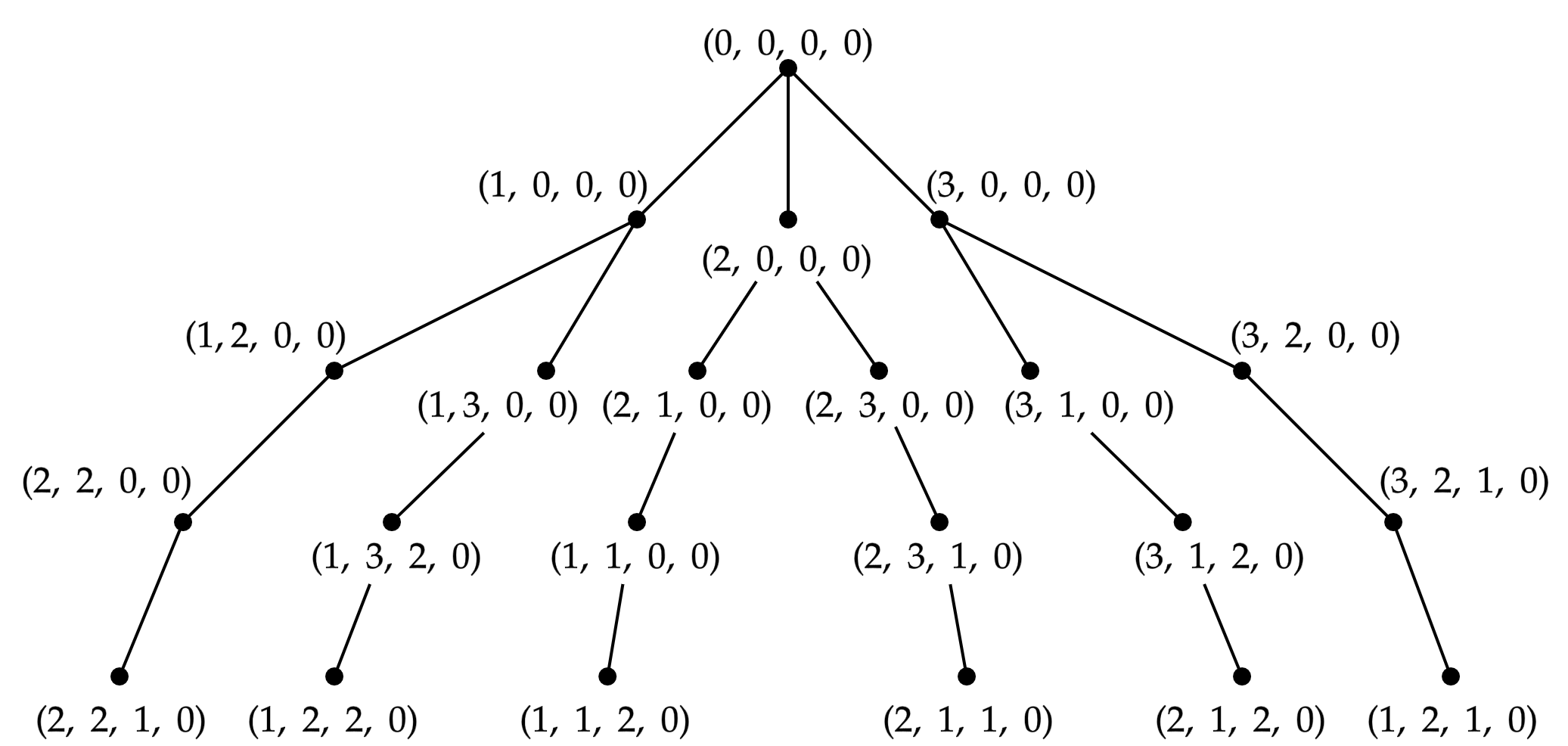}
		\caption{Tree graph of ideal states in a $(3+1)\times (3+1)$ Tower of Hanoi game.}
		\label{fig: treegraph}
	\end{figure}
	
%	\newpage
	Note that from top to bottom the ideal states illustrated in Figure~\ref{fig: 4x4ideal} match precisely with the vectors from left to right of the tree in Figure \ref{fig: treegraph}.
	Each ``level'' in the graph represents one move, i.e., the movement of a single disk.
	Moreover, each node of the graph is labeled by the state denoting the arrangement of the disks on the pegs. 
	The six unique ideal states are represented with vectors at the end of the tree, and exactly four moves are used to get to each of these ideal states.
	
	We are now able to give a characterization of the vectors that encode the ideal states of the Tower of Hanoi.

	\begin{theorem}\label{thm: tower}
		Let $x=(x_0,x_1,\ldots,x_n)\in\{0,1,2,\ldots,n\}^{n+1}$. Then $x\in \ToH_{n+1}$ if and only if $x$ satisfies:
		\begin{enumerate}
			\setcounter{enumi}{-1}
			\item $x_n=0$,
			\item there exists $j\in[n-1]$ and distinct $k,k' \in [n-1]$ with $x_k = x_{k'} = j$, and
			\item if $i \neq k, k'$, then $x_i \neq j$ and 
			$
			\{x_i:0\leq i\leq n-1 \mbox{ and }i\neq k, k'\} 
			=
			[n-1]\setminus\{j\}$. 
		\end{enumerate}
	\end{theorem}
	
	\begin{proof}
		In the forward direction: Recall that if $x$ is an ideal state, then disk $n$ is on peg 0; peg $n$ is empty; and, pegs 1 through $n-1$ each have a single disk, with the exception of one peg having two disks. 
		Suppose that peg $j$ is the peg with two distinct disks $k$ and $k'$.
		In vector notation, the ideal state $x=(x_0,x_1,\ldots,x_n)$, where $x_i$ denotes the peg that disk $i$ is on, would satisfy:
		$x_n=0$,
		$n$ is not an entry of $x$,
		$j$ appears exactly twice in $x$ at indices $k$ and $k'$,
		and $\{x_0,x_1,\ldots,x_{n-1}\}\setminus\{x_k,x_{k'}\}=[n-1]\setminus\{j\}$. 
		These are the required conditions, establishing the forward direction. 
		
		In the opposite direction: Note that condition (0) ensures that disk $n$ lies on peg 0. Conditions (1) and (2) ensure that there exists a unique interior peg containing two disks. 
		Condition (2) also ensures that
		each of the remaining $n-2$ interior pegs has exactly one disk from the set of disks labeled $\{0,1,2,\ldots, n-1\}\setminus\{k,k'\}$. 
		Thus the vector $x$ represents an ideal state in $\ToH_{n+1}$, as desired.
	\end{proof}

	\section{Bijection.}\label{sec: bijection}
	
	We now establish our main theorem. 
	\begin{proof}[Proof of Theorem~\ref{thm:main}.]
		We define a function from $\ToH_{n+1}\subset\{0,1,2,\ldots, n\}^{n+1}$ to $\dPF{n}{1}\subset[n]^n$.
		Let \[f: \ToH_{n+1} \rightarrow PF_n(1)\] be such that if $(x_0,x_1,\ldots,x_{n-1},0)\in \ToH_{n+1}$, with repeated value $j\in[n-1]$, then \[f((x_0,x_1,\ldots,x_{n-1},0))=(a_1,a_2,\ldots,a_n)\] with, for $0\leq i\leq n-1$, \[a_{i+1}=
		\begin{cases}
			x_i + 1 & \text{ if } x_i > j \\
			x_i & \text{ if } x_i \leq j. \\
		\end{cases}
		\]
		It is observed from Theorems \ref{thm: pf} and \ref{thm: tower} that $f$ is injective and surjective.% that $f$ is injective and surjective, which yields the result.
	\end{proof}
	We conclude this work by referencing the result of Klav\v{z}ar, Milutinovi\'{c}, and Petr in \cite[Proposition 2]{klavvzar2005hanoi}, which establishes that the number of ideal states of the Tower of Hanoi is given by the Lah numbers. By Theorem \ref{thm:main} this also gives the number of parking functions with displacement one. 
	\begin{theorem}\label{thm:THcount}
		For $n\geq 1$, $|\ToH_{n+1}|=|\dPF{n}{1}|=\frac{n!(n-1)}{2}$.
	\end{theorem}

	\begin{acknowledgment}{Acknowledgments.}
		Authors Y.~Aguillon, D.~Alvarenga, P.~E. Harris, J.C.~Mart\'{i}nez Mori, C.~Monroe, C.~Tieu, and D.~A.~Williams II were supported by the National Science Foundation Grant No. DMS-1659138 and the Sloan Grant No. G-2020-12592. P.~E.~Harris was supported through a Karen Uhlenbeck EDGE Fellowship.
	\end{acknowledgment}

	\end{document}